\documentclass{amsart}
\usepackage{amssymb}
\usepackage{amsfonts}
\usepackage{amssymb}
\usepackage{amsmath}
\usepackage{amsthm}
\usepackage{enumerate}
\usepackage{tabularx}
\usepackage{centernot}
\usepackage{mathtools}
\usepackage{stmaryrd}
\usepackage{amsthm,amssymb}
\usepackage{etoolbox}
\usepackage{tikz}
\usepackage{amssymb}
\usetikzlibrary{matrix}
\usepackage{tikz-cd}
\usepackage{tikz}
\usepackage{marginnote}
\definecolor{mygray}{gray}{0.85}
\usepackage[backgroundcolor=mygray,colorinlistoftodos,prependcaption,textsize=small]{todonotes}
\usepackage{xargs}                      

\renewcommand{\leq}{\leqslant}
\renewcommand{\geq}{\geqslant}

\makeatletter
\def\subsection{\@startsection{subsection}{3}%
  \z@{.5\linespacing\@plus.7\linespacing}{.3\linespacing}%
  {\bfseries\centering}}
\makeatother

\makeatletter
\def\subsubsection{\@startsection{subsubsection}{3}%
  \z@{.5\linespacing\@plus.7\linespacing}{.3\linespacing}%
  {\centering}}
\makeatother

\makeatletter
\def\myfnt{\ifx\protect\@typeset@protect\expandafter\footnote\else\expandafter\@gobble\fi}
\makeatother

\newtheorem{theorem}{Theorem}

\newtheorem{definition}[theorem]{Definition}
\newtheorem{lemma}[theorem]{Lemma}

\newtheorem{observation}[theorem]{Observation}
\newtheorem{fact}[theorem]{Fact}
\newtheorem{conclusion}[theorem]{Conclusion}

\newtheorem{notation}[theorem]{Notation}

\newtheorem{convention}[theorem]{Convention}

\newtheorem*{nproblem}{Problem}

\newtheorem*{nconjecture}{Conjecture}
\newtheorem*{maintheorem}{Main Theorem}
\newtheorem*{nfact}{Fact}

\newcounter{claimcounter}
\numberwithin{claimcounter}{theorem}

\begin{document}

\begin{abstract} In \cite{tzaban}, given a metrizable profinite group $G$, a cardinal invariant of the continuum $\mathfrak{fm}(G)$ was introduced, and a positive solution to the Haar Measure Problem for $G$ was given under the assumption that $\mathrm{non}(\mathcal{N}) \leq \mathfrak{fm}(G)$. We prove here that it is consistent with ZFC that there is a metrizable profinite group $G_*$ such that $\mathrm{non}(\mathcal{N}) > \mathfrak{fm}(G_*)$, thus demonstrating that the strategy of \cite{tzaban} does not suffice for a general solution to the Haar Measure Problem.
\end{abstract}

\title{On a Cardinal Invariant Related to the Haar Measure Problem}
\thanks{Partially supported by European Research Council grant 338821. No. 1148 on Shelah's publication list. The present paper was written while the first author was a post-doc research fellow at the Einstein Institute of Mathematics of the Hebrew University of Jerusalem, supported by European Research Council grant 338821.}

\author{Gianluca Paolini}
\address{Department of Mathematics ``Giuseppe Peano'', University of Torino, Via Carlo Alberto 10, 10123, Italy.}
\email{gianluca.paolini@unito.it}

\author{Saharon Shelah}
\address{Einstein Institute of Mathematics,  The Hebrew University of Jerusalem, Israel \and Department of Mathematics,  Rutgers University, U.S.A.}

\date{\today}
\maketitle

\section{Introduction}

	It is well-known that every compact group admits a unique translation-invariant probability measure, its {\em Haar measure}. A long-standing\footnote{The problem dates back at least to 1963, when in \cite[Section 16.13(d)]{hewitt} the problem was posed and settled in the positive in the abelian case.} open problem asks:
	
	\begin{nproblem}[Haar Measure Problem] Does every infinite compact group have a non-Haar-measurable subgroup?
\end{nproblem}

	In \cite{morris} the problem was settled in the positive under the assumption that the compact group is not an infinite {\em metrizable profinite group}. Furtheremore, in \cite{mislove} it was proved that it is consistent with ZFC that every infinite compact group has a non-Haar-measurable subgroup. Very recently, progress  has been made toward a solution to the Haar measure problem for infinite metrizable profinite groups. In fact, in \cite{tzaban} the authors introduced a certain cardinal invariant of the continuum $\mathfrak{fm}(G)$, depending on a metrizable profinite group $G$, \mbox{and proved (see Sec. \ref{sec_pre} for definitions):}
	
	\begin{nfact}[\cite{tzaban}] Let $G$ be an infinite metrizable profinite group. If $\mathrm{non}(\mathcal{N}) \leq \mathfrak{fm}(G)$, then $G$ has a non-Haar measurable subgroup.
\end{nfact}

	Also in \cite{tzaban}, the authors conjectured:
	
	\begin{nconjecture}[\cite{tzaban}] Let $G$ be an infinite metrizable profinite group. Then:
	$$\mathrm{non}(\mathcal{N}) \leq \mathfrak{fm}(G).$$
\end{nconjecture}

	In this work we refute the conjecture above, thus demonstrating that the strategy of \cite{tzaban} does not suffice for a general solution to the Haar Measure Problem.
	
	\begin{maintheorem}\label{main_th} It is consistent with ZFC that there exists an infinite metrizable profinite group $G_*$ such that:
	$$\mathrm{non}(\mathcal{N}) > \mathfrak{fm}(G_*).$$
\end{maintheorem}

	Notice that in the aforementioned work from \cite{mislove}, the exibithed models of ZFC witnessing that the Haar Measure Problem has consistently a positive answer do {\em not} satisfy $\mathrm{CH}$, while, despite the failure of the main conjecture in \cite{tzaban} proved in this paper, the work of \cite{tzaban} shows the remarkable result that in {\em all} the models of ZFC satisfying $\mathrm{CH}$ the Haar Measure Problem has a positive answer.

\section{Preliminaries}\label{sec_pre}

	\begin{convention} 
	\begin{enumerate}[(1)]
	\item We denote by $\omega$ the set of natural numbers.
	\item Given $n < \omega$, we identify $n$ with the set $\{ 0, ..., n-1 \} = [0, n)$.
	\item Given a set $X$ we denote by $\mathcal{P}(X)$ the set of subsets of $X$.
	\item Given a set $X$ and $n < \omega$, we denote by $[X]^n$ the set of subsets of $X$ of power~$n$.
\end{enumerate}
\end{convention}

	\begin{definition}\label{def_metr_pro} A {\em metrizable profinite group} $G$ is a profinite group of the form $\varprojlim^{\bar{\varphi}}_{i < \omega} G_i$, for $\bar{\varphi} = (\varphi_i : i < \omega)$ and $\varphi_i \in \mathrm{Hom}(G_{i+1}, G_i)$, i.e. $G$ is an inverse $\bar{\varphi}$-limit of an $(\omega, <)$-inverse system of finite groups. When the homorphisms $\varphi_i$ are clear from the context, we might forget to mention $\bar{\varphi}$ and simply write $\varprojlim_{i < \omega} G_i$.
\end{definition}

	\begin{notation}\label{notation_measure} Given a metrizable profinite group we denote by $\mu$ its Haar measure, i.e. the unique translation-invariant probability measure defined on $G$.
\end{notation}

	\begin{notation} Let $1 < n < \omega$, $A \subseteq G^n$ and $g \in G$. We let:
	$$A_g = \{ (h_1, ..., h_{n-1}) \in G^{n-1} :  (h_1, ..., h_{n-1}, g)\in A \}. $$
\end{notation}

	\begin{definition}\label{alg_set_def} Let $G$ be a metrizable profinite group. 
	\begin{enumerate} [(1)]
	\item We say that $X \subseteq G^{n}$ is an {\em elementary algebraic set} if there is a group word $w(\bar{x}, \bar{z})$, with $|\bar{x}| = n$, and a sequence of parameters  $\bar{c} \in G^{|\bar{z}|}$ such that: 
	$$X = \{ \bar{a} \in G^{|\bar{x}|} : G \models w(\bar{a}, \bar{c}) = e \}.$$
	\item We say that $X \subseteq G^{n}$ is an {\em elementary algebraic null set} if $X$ is an elementary algebraic set which is null with respect to $\mu$ (cf. Notation \ref{notation_measure}).
	\item\label{fub_mark_def} We say that $X \subseteq G$ is {\em Fubini-Markov} if either of the following happens:
	\begin{enumerate}[(a)]
	\item $X$ is an elementary algebraic null set;
	\item there is $1 < n < \omega$ and an elementary algebraic null set $A \subseteq G^n$ such that: $$X = \{ g \in G : \mu(A_g) > 0 \}.$$
	\end{enumerate}
\end{enumerate} 
\end{definition}

	\begin{definition}\label{fm_card} Let $G$ be a metrizable profinite group. The cardinal invariant $\mathfrak{fm}(G)$ is the smallest size of a collection of Fubini-Markov sets whose union has measure~$1$.
\end{definition}



	\begin{fact}\label{description_measure} Let $G = \varprojlim^{\bar{\varphi}}_{i < \omega} G_i$ be a metrizable profinite group and let $\pi_i$ be the canonical projection of $G$ onto $G_i$, for $i < \omega$. Let $U \subseteq G$ be a closed set of the form:
	$$U = \bigcap_{i < \omega} \pi_i^{-1}(B_i),$$
with $B_i \subseteq G_i$ and $\varphi_{i}(B_{i+1}) = B_i$, for $i < \omega$. Then:
	$$\mu(U) = \lim_{i \rightarrow \infty} \frac{|B_i|}{|G_i|}.$$
\end{fact}

	\begin{proof} Notice that:
	$$\begin{array}{rcl}\label{equation}
	\mu(U)& = & \mu(\bigcap_{i < \omega} \pi_i^{-1}(B_i)) \\
		  & = & lim_{i \rightarrow \infty} \mu(\pi^{-1}(B_i)) \;\;\;\;\; (\text{by \cite[Chapter 18, item 2f, pg. 363]{field_arithmetic}}) \\
		  & = & lim_{i \rightarrow \infty} \frac{|B_i|}{|G_i|} \;\;\;\;\;\;\;\;\;\;\;\;\;\;\;\;\, (\text{by \cite[Chapter 18, Example 18.2.3]{field_arithmetic}}).
\end{array}$$
\end{proof}

\begin{definition}\label{non_null} We denote by $\mathcal{N}$ the ideal of null sets in the Cantor space $2^\omega$, and by $non(\mathcal{N})$ the minimal cardinality of a non-null subset of $2^\omega$.
\end{definition}


\section{Building Appropriate Finite Groups}

	\begin{notation}\label{not_prod} Let $G$ be a group and $\bar{g} = (g_i : i < n)$, for $n < \omega$, a finite sequence of elements of $G$. Given $I \subseteq n$ we let  $g_I = \prod_{i \in I} g_i \in G$ (if $I = \emptyset$, then $g_I = e$).
\end{notation}

	\begin{definition}\label{CR} For $2 \leq 4m \leq n < \omega$ such that $\frac{2}{2^m} + \frac{1}{n^2} < \frac{1}{m}$, let $\mathbf{CR}_{(n, m)}$ be the class of triples $(G, \bar{y}, \bar{z})$ such that:
	\begin{enumerate}[(a)]
	\item $G$ is a finite group;
	\item\label{the_y} $\bar{y} = (y_i : i < n)$ is a sequence of pairwise commuting elements of $G$ each of order $2$ and such that $\langle \bar{y} \rangle_G$ is a subgroup of order $2^n$;
	\item $\bar{z} = (z_I : I \in [n]^m)$ and $z_I \in G$;
	\item\label{item_zcommutesy} for every $I \subseteq n$ and $J \in [n]^m$, $[y_I, z_J] = e$ iff $I \in \{ J, \emptyset \}$ (cf. Notation~\ref{not_prod});
	\item if $s \in G - \{ e \}$, then $|\{ t \in G : [s, t] = e \}| < |G|/n^2$.
	\end{enumerate}
\end{definition}

	\begin{lemma}\label{pre_finite_groups} For $n, m < \omega$ as in Definition \ref{CR}, $\mathbf{CR}_{(n, m)} \neq \emptyset$ (cf. Definition \ref{CR}).
\end{lemma}

\begin{proof} Let $G_0$ be the Abelian group $\bigoplus \{ \mathbb{Z}_2 y_i : i < 2n \}$ (where $\mathbb{Z}_2 y_i$ is the group with two elements with generator $y_i$), and, for $I \subseteq n$, let $y_I = \sum \{ y_i : i \in I \}$ (i.e. we are using Notation \ref{not_prod} in additive notation). For $I \subseteq n$, let $\pi_I \in Aut(G_0)$ be such that for every $J \subseteq n$ with $J \notin \{ \emptyset, I \}$ we have that:
$$\pi_I(y_J) \neq y_J \;\; \text{ and } \;\; \pi_I(y_I) = y_I.$$
[Why must such $\pi_I$'s exist? Let $(y^I_\ell : \ell < 2n)$ be a basis of $G_0$ such that $y^I_0 = y_I$, if $I \neq \emptyset$, and any $x \in G_0 - \{ e \}$ otherwise (it is well known that every $x \in G_0 - \{ e\}$ can be extended to a basis of $G_0$). Let $\pi'_I$ be such that, $\pi'_I(y^I_\ell) = y_{n+\ell}$ , for $\ell \in (0, n)$, and $\pi'_I(y^I_0) = y^I_0$. Then any extension of $\pi'_I$ to a $\pi_I \in Aut(G_0)$ is as wanted.]
\newline Let $G_1$ be the group generated by $G_0 \cup \{  z_I: I \in [n]^m \}$ freely except for:
	\begin{enumerate}[(i)]
	\item the equations of $G_0$;
	\item if $I \subseteq n$ and $x \in G_0$, then $z_I^{-1} x z_I = \pi_I(x)$.
\end{enumerate}
Let $G$ be $Sym(G_1)$ (the group of permutations of the set $G_1$), interpreting $G_1$ as a subgroup of $G$, and let $\mathbf{n} = |G_1|$. Then clearly $\mathbf{n} > n^2$ (which will be used at the end of the proof). Now, we claim that $(G, \bar{y}, \bar{z}) \in \mathbf{CR}_{(n, m)}$, for $\bar{y} = (y_i : i < n)$ and $\bar{z} = (z_I : I \in [n]^m)$. Clearly, clauses (a)-(d) of Definition \ref{CR} hold.
Finally, concerning condition (e), notice that if $s \in G - \{ e \}$, then:
$$|\{ t \in G : [s, t] = e \}|  \leq \frac{\mathbf{n}!}{(\mathbf{n}-1)!} = \mathbf{n} \leq (\mathbf{n}-1)! =  |G|/\mathbf{n} < |G|/n^2.$$
\end{proof}

	\begin{definition}\label{the_finite_groups} Let $\mathbf{CR}$ be the set of tuples $\mathbf{p}$ such that:
	$$
\begin{array}{rcl}\label{equation}
\mathbf{p}& = & (k_{\mathbf{p}}, m_{\mathbf{p}}, n_{{\mathbf{p}}}, (G_{({\mathbf{p}}, 1)}, \bar{y}^1, \bar{z}^1), G_{({\mathbf{p}}, 2)}) \\
		  & = & (k, m, n, (G_1, \bar{y}^1, \bar{z}^1), G_2),
\end{array}$$
and:
	\begin{enumerate}[$(*)_{0}$]
	\item \begin{enumerate}[(a)]
	\item $0 < k < m < n < \omega$;
	\item $2 \leq 4m \leq n$;
	\item $2^{k}m = n$ and $k << n$;
	\item\label{fractions} $\frac{2}{2^m} + \frac{1}{n^2} < \frac{1}{m}$.
\end{enumerate}
\end{enumerate}
	\begin{enumerate}[$(*)_1$]
	\item $(G_1, \bar{y}^1, \bar{z}^1) \in \mathbf{CR}_{(n, m)}$ (cf. Definition \ref{CR});
\end{enumerate}
	\begin{enumerate}[$(*)_2$]
	\item
	\begin{enumerate}[(a)]
	\item we let $\mathfrak{c}_{\mathbf{p}} = \mathfrak{c} : n \times n \rightarrow G_1$ be such that for $i_0, i_1 < n$ we have:
	\begin{enumerate}[($\alpha$)]
	\item $\mathfrak{c}(i_0, i_1) = e$, if $i_0 \neq i_1$;
\end{enumerate}
	\begin{enumerate}[($\beta$)]
	\item $\mathfrak{c}(i_0, i_1) := y^1_i$, if $i_0 = i_1 = i$;
	\end{enumerate}
	\item $G_2$ is the group generated freely by $G_1 \cup \{ y^{\ell}_i = y_{(\ell, i)} : \ell \in \{ 2, 3 \}, i < n \}$ except for:
	\begin{enumerate}[($\alpha$)]
	\item the equations of $G_1$;
\end{enumerate}
	\begin{enumerate}[($\beta$)]
	\item $y^{\ell}_i$ has order $2$, for every $\ell \in \{ 2, 3 \}$ and $i < n$;
	\end{enumerate}
	\begin{enumerate}[($\gamma$)]
	\item $y^\ell_{i}$ and $y^\ell_{j}$ commute, for every $\ell \in \{2, 3 \}$ and $i, j < n$;
	\end{enumerate}
	\begin{enumerate}[($\delta$)]
	\item for every $\ell \in \{ 2, 3\}$, $i < n$ and $g \in G_1$, $y^\ell_i$ commutes with $g$;
	\end{enumerate}
	\begin{enumerate}[($\epsilon$)]
	\item $[y^2_{i}, y^3_{j}] = \mathfrak{c}(i, j)$, for every $i,  j < n$.
	\end{enumerate}
	\end{enumerate}
	\end{enumerate}
\end{definition}

	\begin{notation}\label{notation_n2} For uniformity of notation, given the context of Definition \ref{the_finite_groups}, and in particular $k$, $m$ and $n$ as there, we will let $n = n_2 = n_3$.
\end{notation}

\begin{lemma}\label{claim_a8} Let $\mathbf{p} \in \mathbf{CR}$ (cf. Definition \ref{the_finite_groups}). Then:
	\begin{enumerate}[(1)]
	\item $G_2 = G_{(\mathbf{p}, 2)}$ is finite, $G_1$ is a normal subgroup of $G_2$ and $G_2/G_1$ is Abelian.
	\item\label{item_words} for every $x \in G_2$, there are unique $\mathcal{U}_{\ell} = \mathcal{U}({\ell}) = \mathcal{U}_{\ell}(x) = \mathcal{U}({\ell}, x) \subseteq [0, n_{\ell})$ (cf. Notation \ref{notation_n2}), for $\ell \in \{ 2, 3 \}$, and $y_{(1, x)} \in G_1$, such that:
	$$x = y_{(3, \mathcal{U}(3))}
		  y_{(2, \mathcal{U}(2))}
		  y_{(1, x)},$$
where, for $\ell \in \{ 2, 3 \}$, we let:
	$${y}_{(\ell, \mathcal{U}({\ell}))} = \prod_{i \in \mathcal{U}({\ell})} y^\ell_i.$$
	\end{enumerate}
\end{lemma}

	\begin{proof} Clear.
\end{proof}

	\begin{lemma}\label{partitions} Let $\mathbf{p} \in \mathbf{CR}$ (cf. Definition \ref{the_finite_groups}), $G_2 = G_{(\mathfrak{p}, 2)}$, and $k = k_{\mathbf{p}}$. If $x_0, ..., x_{k-1} \in G_2$, then for some $I_* \subseteq [0, n_2)$ (cf. Notation \ref{notation_n2}) we have:
	\begin{enumerate}[(a)]
	\item $|I_*| = n_2/2^{k}$ (recall that $n_2/2^{k} = n/2^{k} = 2^km/2^k= m$);
	\item if $\ell < k$, then $\mathcal{U}_2(x_\ell) \cap I_* \in \{I_*, \emptyset \}$ (cf. Lemma \ref{claim_a8}(\ref{item_words})).
	\end{enumerate}
\end{lemma}

	\begin{proof} For $\eta \in 2^k$, let:
	$$I_{\eta} = \{ i < n_2 : \text{ if } \ell < k, \text{ then } i \in \mathcal{U}_2(x_\ell) \Leftrightarrow \eta(\ell) = 1 \}.$$
So $(I_{\eta} : \eta \in 2^k)$ is a partition of $[0, n_2)$ into $2^k$ parts, hence for some $\eta \in 2^k$ we have that $|I_{\eta}| \geq n_2/2^k$ (recall that $2^{k} \mid n_2$ and $k << n_2$). Now, let $I_* \subseteq I_{\eta}$ be such that it satisfies clause (a) of the statement of the lemma. Then $I_*$ is as wanted.
\end{proof}

\begin{lemma}\label{crucial_lemma} Let $\mathbf{p} \in \mathbf{CR}$ (cf. Definition \ref{the_finite_groups}). If $x_\ell \in G_2 = G_{(\mathbf{p}, 2)}$, for $\ell < k = k_{\mathbf{p}}$, then for some $I_* \subseteq n$ and $c, c_* \in G_2$ we have:
	\begin{enumerate}[(a)]
	\item $c = y^3_{I_*}$ and $c_* = z^1_{I_*}$;
	\item $G_2 \models [[x_\ell, c], c_*] = e$;
	\item $|I_*| = n_2/2^k$;
	\item $(B_I : I \subseteq I_*)$ is a partition of $G_2$ into sets of equal size such that:
	$$G \models [[x, c], c_*] = e \text{ iff } x \in B_{\emptyset} \cup B_{I_*},$$
where, for $I \subseteq I_*$, we let:
	$$B_I = \{ a \in G_2 : [a, c] = y^1_I \};$$
	\item 
	$$|\{ (x, y) \in G_2 \times G_2 : G_2 \models [[[x, c], c_*], y] = e \}| \leq \frac{|G_2 \times G_2|}{m}.$$
	\end{enumerate}
\end{lemma}

	\begin{proof} Let $x_\ell \in G_2$, for $\ell < k$, and let $I_* \subseteq [0, n_2)$ be as in Lemma \ref{partitions} with respect to $(x_0, ..., x_{k-1})$.
Let $c = \prod \{ y^3_i : i \in I_* \} = y_{(3, I_*)}$ and $c_* = z^1_{I_*}$ (cf. Definitions \ref{CR} and \ref{the_finite_groups}). We have to show that $(I_*, c, c_*)$ are as wanted. To this extent, let $a \in G_2$  and let:
			$$a = y_{(3, \mathcal{U}({3}))}
		   	y_{(2, \mathcal{U}({2}))}
		   	y_{(1, a)}$$ 
be as in Lemma~\ref{claim_a8}(\ref{item_words}), for $\mathcal{U}({\ell}) = \mathcal{U}({\ell}, a) \subseteq [0, n_{\ell})$, and $\ell \in \{ 2, 3 \}$. Notice that for $\ell \in \{ 2, 3 \}$ and $I_\ell \subseteq [0, n_\ell)$ we have that $(y^\ell_{I_\ell})^{-1} = y^\ell_{I_\ell}$ (cf. Notation \ref{not_prod}), since each element of the product has order $2$ and they all commute with each other. Then for any $a \in G_2$ we have that (recalling Lemma \ref{claim_a8} and letting $y_{(\ell, \mathcal{U}({\ell}))} = y_{(\ell, \mathcal{U}({\ell}, a))}$):
$$\begin{array}{rcl}
[a, c]& = & a^{-1} c^{-1} a c \\
    	& = & 
		   	  (y_{(1, a)})^{-1}
		   	  y_{(2, \mathcal{U}({2}))}
		   	  y_{(3, \mathcal{U}({3}))}  y_{(3, I_*)} y_{(3, \mathcal{U}({3}))}
		   	  y_{(2, \mathcal{U}({2}))}
		   	  y_{(1, a)} y_{(3, I_*)}\\
		& = &
		   	  y_{(2, \mathcal{U}({2}))}
		   	  y_{(3, \mathcal{U}({3}))}  y_{(3, I_*)} y_{(3, \mathcal{U}({3}))}
		   	  y_{(2, \mathcal{U}({2}))}
		   	   y_{(3, I_*)}\\
		& = &
		      y_{(2, \mathcal{U}({2}))}
		   	  y_{(3, I_*)}
		   	  y_{(2, \mathcal{U}({2}))} \hat{y}_{(3, I_*)} \;\;\;\;\;\;\;\;\;\;\;\;\;\;\;\;\;\;\;\;\;\;\;\;\;\;\;\;\;\;\; [\text{by \ref{the_finite_groups}$(*)_2(b)$($\beta$)-($\gamma$)}]\\
		& = & y_{(2, \mathcal{U}({2}) \cap I_*)}
		   	  y_{(3, I_*)}
		   	  y_{(2, \mathcal{U}({2}) \cap I_*)} y_{(3, I_*)} \;\;\;\;\;\;\;\;\;\;\;\;\;\;\;\;  [\text{by \ref{the_finite_groups}$(*)_2(a)$($\beta$)+$(b)$($\epsilon$)}]\\
		& = & y_{(2, \mathcal{U}({2}) \cap I_*)}
		   	  y_{(3, \mathcal{U}({2}) \cap I_*)}
		   	  y_{(2, \mathcal{U}({2}) \cap I_*)} y_{(3, \mathcal{U}({2}) \cap I_*)}
		   	  \, [\text{by \ref{the_finite_groups}$(*)_2(a)$($\beta$)+$(b)$($\epsilon$)}]\\
		& = & \prod_{i \in \mathcal{U}({2}) \cap I_*} \mathfrak{c}_2(i, i) \hspace{-0.020cm}\,\;\;\;\;\;\;\;\;\;\;\;\;\;\;\;\;\;\;\;\;\;\;\;\;\;\;\;\;\;\;\;\;\;\;\;\;\;\;\;\;\;\;\;\;\;\;\;\;\;\;\;\;\;[\text{by \ref{the_finite_groups}$(*)_2(b)$($\epsilon$)}]\\
		& = & y^1_{\mathcal{U}({2}) \cap I_*} \hspace{-0.08cm}\,\;\;\;\;\;\;\;\;\;\;\;\;\;\;\;\;\;\;\;\;\;\;\;\;\;\;\;\;\;\;\;\;\;\;\;\;\;\;\;\;\;\;\;\;\;\;\;\;\;\;\;\;\;\;\;\;\;\;\;\;\;\;\;\;\;\;\;\, [\text{by \ref{the_finite_groups}$(*)_2(a)$($\beta$)}]\\
		& = & y^1_{\mathcal{U}({2}, a) \cap I_*}.		
		\end{array} $$
Hence, recapitulating, we have:
\begin{equation}\tag{$\star$}
[a, c] = y^1_{\mathcal{U}({2}, a) \cap I_*}.
\end{equation}
Concerning clause (b), by Equation $(\star)$ for $a = x_\ell$, Lemma~\ref{partitions} and the fact that the triple $(G_{(\mathbf{p}, 1)}, \bar{y}^1, \bar{z}^1) \in \mathbf{CR}_{(n, m)}$ we have that $[x_{\ell}, c] = e$ or $[x_{\ell}, c] = y^1_{I_*}$, and in both cases $[x_{\ell}, c]$ commutes with $z^1_{I_*} = c_*$ (cf. Definition \ref{CR}(\ref{item_zcommutesy})). Clause (c) holds by Lemma~\ref{partitions}, since by choice $|I_*| = n_2/2^{k}$. As for clause (d), clearly, the $(B_I : I \subseteq I_*)$ are pairwise disjoint, since $a \in B_{I_1} \cap B_{I_2}$ implies $y^1_{I_1} = [a, c] = y^1_{I_2}$, and for $I_1 \neq I_2$ we have that $y^1_{I_1} \neq y^1_{I_2}$ (cf. Definition \ref{CR}(\ref{the_y})); moreover, by Equation $(\star)$, if $a \in G_2$, then $[a, c] = y^1_{\mathcal{U}({2}, a) \cap I_*} \in \{ y^1_I : I \subseteq I_*\}$, and for $I \subseteq I_*$ we have that $[y^1_I, y^1_{I_*}] = e$ if and only if $I \in \{ \emptyset, I_* \}$ (cf. Definition \ref{CR}(\ref{item_zcommutesy})); and finally the pieces of the partition are of equal size since given a finite set $X$, a subset $Y$ of $X$ and two subsets $c_1$ and $c_2$ of $Y$ we have that $|\{ Z \subseteq X : Z \cap Y = c_1 \}| = |\{ Z \subseteq X : Z \cap Y = c_2 \}|$.
Concerning clause (e), let:
\begin{enumerate}[(a)]
	\item $X = \{ (x, y) \in G_2 \times G_2 : [[[x, c], c_*], y] = e \}$;
	\item $X_1 = \{ (x, y) \in G_2 \times G_2 : [x, c] \in \{ y^1_{I_*}, e \} \}$;
	\item $X_2 = \{ (x, y) \in X : [x, c] \in \{ y^1_{I} : I \subseteq I_*, I \notin \{ I_*, \emptyset\} \} \};$
\end{enumerate} 
Clearly $X = X_1 \cup X_2$ and $X_1 \cap X_2 = \emptyset$.
Now, on one hand, we have:
\begin{equation}\label{equation1}
|X_1| \leq |G_2 \times G_2| \cdot \frac{|\{ \emptyset, I_*\}|}{2^{|I_*|}} = |G_2 \times G_2| \cdot \frac{2}{2^{|I_*|}}.
\end{equation}
While, on the other hand, we have:
\begin{equation}\label{equation2}
	|X_2| \leq \frac{|G_2 \times G_2|}{n^2}.
\end{equation}
[Why does (\ref{equation2}) hold? First of all notice that:
\begin{enumerate}[$\oplus_1$]
	\item if $x \in B_I$, $\mathcal{U}(2, x) \cap I_* = I \subseteq I_*$, $I \notin \{ I_*, \emptyset \}$, then:
	\begin{enumerate}[(a)]
	\item $[[x, c], c_*] \neq e$ (by clause (d) of the current lemma);
	\item $[[x, c], c_*] \in G_1$ (because by $(\star)$ $[x, c] = y^1_{\mathcal{U}({2}, x) \cap I_*} \in G_1$, and $c_* = z^1_{I_*} \in G_1$).
	\end{enumerate}
\end{enumerate}
Secondly, notice that:
\begin{enumerate}[$\oplus_2$]
	\item \begin{enumerate}[(a)]
	\item if $t = G_1 - \{e \}$, then:
	$$\begin{array}{rcl}
Z_t   & := & \{ x \in G_2 : [t, x] = e \} \\
      & = & \{ x \in G_2 :  x = y_{(3, \mathcal{U}(3))}
		  y_{(2, \mathcal{U}(2))}
		  y_{(1, x)} \text{ and } [y_{(1, x)}, t] = e \} \; (\text{cf. Lemma \ref{claim_a8}}); 		
	  \end{array} $$
	\item and so for $t = G_1 - \{e \}$ we have:
	$$\begin{array}{rcl}
|Z_t| & \leq & 2^{n_3} \cdot 2^{n_2} \cdot |\{ y_1 \in G_1 : [y_1, t] = e \}|\\
      & \leq & |G_2| \cdot \frac{1}{|G_1|} \cdot max_{t \in G_1 - \{ e \}} |\{ y_1 \in G_1 : [y_1, t] = e \}|; 
	  \end{array} $$
	\item and thus, by (b) and Definition \ref{CR}(e), we have:
	$$t \in G_1 - \{ e \} \; \Rightarrow \; |Z_t| \leq |G_2| \cdot \frac{1}{n^2}.$$
	\end{enumerate}
\end{enumerate}
Hence, we have:
$$ \begin{array}{rcl}
|X_2| & \leq & |G_2| \cdot max_{\substack{x \in G_2 \\ \mathcal{U}(2, x) \cap I_* \notin  \{\emptyset, I_* \}}} |\{ y \in G_2 : [[[x, c], c_*], y] = e \}| \\
      & \leq & |G_2| \cdot max_{t \in G_1 - \{ e \}} |\{ y \in G_2 : [y, t] = e \}| \;\;\;\; \text{ [by $\oplus_1$]}\\
      & \leq & \displaystyle \frac{|G_2 \times G_2|}{n^2} \;\;\;\;  \text{ [by $\oplus_2$(c)].}		
		\end{array} $$
That is, Equation (\ref{equation2}) holds as promised. This closes the ``Why (2)?'' above.]
\newline Hence, putting together (\ref{equation1}) and (\ref{equation2}) we have:
$$|\{ (x, y) \in G_2 \times G_2 : G_2 \models [[[x, c], c_*], y] = e \}| \leq |G_2 \times G_2| \cdot (\frac{2}{2^{|I_*|}} + \frac{1}{n^2}) \leq \frac{|G_2 \times G_2|}{m},$$
by the choice of $m$ and $n$, in fact by (c) of this lemma we have that $|I_*| = n_2/2^{k}$ and, by Definition (\ref{the_finite_groups})$(*)_{0}$(d) and Notation \ref{notation_n2}, $n_2/2^{k} = n/2^{k} = 2^km/2^k= m$.
\end{proof}

	\begin{conclusion}\label{crucial_conclusion} Assume that $\mathbf{p} \in \mathbf{CR}$ (cf. Definition \ref{the_finite_groups}). If $x_\ell \in G_2 = G_{(\mathbf{p}, 2)}$, for $\ell < k = k_{\mathbf{p}}$, then for some $c_1, c_2 \in G_2$ we have:
	\begin{enumerate}[(a)]
	\item $G_2 \models [[x_\ell, c_1], c_2] = e$;
	\item $\{ y \in G_2 : G_2 \models [[[x_\ell, c_1], c_2], y] = e \} = G_2$;
	\item $|\{ (x, y) \in  G_2 \times G_2 : G_2 \models [[[x, c_1], c_2], y] = e\}| \leq |G_2 \times G_2|/m$.
\end{enumerate}	 
\end{conclusion}

	\begin{proof} This is clear from Lemma \ref{crucial_lemma} letting $c_1 = c$ and $c_2 = c_*$, for $c, c_*$ as there.
\end{proof}

\section{The Solution}

	\begin{notation}\label{notation_solution} (Recalling the notation of Definition \ref{the_finite_groups}.) We choose $(f_1, g_1)$ and $(f_2, g_2)$ such that:
	\begin{enumerate}[(a)]
	\item $f_1, g_1, f_2, g_2$ are strictly increasing functions from $\omega^\omega$;
	\item $f_\ell(n) > g_\ell(n)$, for $\ell \in \{ 1, 2 \}$ and $n < \omega$;
	\item $(f_1, g_1)$ and $(f_2, g_2)$ are sufficiently different (as in \cite{forcing}), e.g., for every $i < \omega$ we have $2^{2^{f_1(i)}} < g_2(i)$ and $2^{2^{f_2(i)}} < g_2(i+1)$;
	\item for every $i < \omega$, there is $\mathbf{p}_i \in \mathbf{CR}$ (cf. Definition \ref{the_finite_groups}) such that:
	\begin{enumerate}
	\item $f_1(i) = |G_{(\mathbf{p}_i, 2)}|$;
	\item $g_{2}(i) = k_{\mathbf{p}_i}$;
	\end{enumerate}
	\item $\sum_{i < \omega} \frac{g_2(i)}{f_2(i)} < \infty$;
	\item\label{f} for $i < \omega$, let $(m^*_i, m^{**}_i) = (g_2(i), f_2(i))$;
	\item for $i < \omega$, let $k_{\mathbf{p}_i} = k_i$, $m_{\mathbf{p}_i} = m_i$, $n_{\mathbf{p}_i} = n_i$ and $G^*_i = G_{(\mathbf{p}_i, 2)}$;
	\item let $G_* = \prod_{i < \omega} G^*_i$.
\end{enumerate}
\end{notation}

	\begin{observation}
	\begin{enumerate}[(1)]
	\item For every $i < \omega$, $G^*_i$ is a finite group.
	\item $G_*$ is a metrizable profinite group (cf. Definition \ref{def_metr_pro}).
	\end{enumerate}
\end{observation}

	\begin{proof} Item (1) is by Lemma \ref{claim_a8}. Item (2) is by definition.
\end{proof}

\begin{notation} 
	\begin{enumerate}[(1)]
	\item We denote by $w(x, y, \bar{z})$, for $\bar{z} = (z_1, z_2)$, the group word:
	$$[[[x, z_1], z_2], y].$$
From now till the end of the paper the letter $w$ will denote this specific word.
	\item Recall Notation \ref{notation_measure}, i.e. we denote by $\mu$ the Haar measure.
\end{enumerate}
\end{notation}

	\begin{notation}\label{notation_for_crucial_lemma}
	\begin{enumerate}[(1)]
	\item For $\bar{c} \in G_* \times G_*$, let:
	 $$X_{\bar{c}} = \{ x \in G_* : \mu(\{ y \in G_* : w(x, y, \bar{c}) \}) > 0 \}.$$
	\item Let $\mathfrak{C} = \{ \bar{c} \in G_* \times G_*: \mu(\{ (x, y) \in G_* \times G_* : w(x, y, \bar{c}) \}) = 0 \}$.
	\end{enumerate}
\end{notation}

	\begin{lemma}\label{cruc_lemma1} A sufficient condition for $\mathfrak{fm}(G_*) \leq \lambda$ (cf. Definition \ref{fm_card}) is:
\begin{enumerate}[$(\star)_1$]
	\item there is $\mathcal{F} \subseteq \prod_{i < \omega} [G^*_i]^{k_i}$ of cardinality $\leq \lambda$ such that:
	\begin{equation} \tag{A} (\forall \eta \in \prod_{i < \omega} G^*_i) (\exists \nu \in \mathcal{F})[\eta(i) \in \nu(i)].
\end{equation}
\end{enumerate}
\end{lemma}

	\begin{proof} For every $\nu \in \mathcal{F}$ and $i < \omega$, $\nu(i) \in [G^*_i]^{k_i}$, hence, by Conclusion \ref{crucial_conclusion}, there are $c^{\nu}_{i, 1}, c^{\nu}_{i, 2} \in G^*_i \times G^*_i$ such that letting $\bar{c}^\nu_i = (c^{\nu}_{i, 1}, c^{\nu}_{i, 2})$ we have:
	\begin{enumerate}[(a)]
	\item if $x \in \nu(i)$, then $|\{ y \in G^*_i: w(x, y, \bar{c}^\nu_i) = e \}| = |G^*_i|$;
	\item $|\{ (x, y) \in G^*_i \times G^*_i : w(x, y, \bar{c}^\nu_i) = e\}| \leq |G^*_i \times G^*_i| / m$.
	\end{enumerate}
Let now $\bar{c}_\nu = (\bar{c}_{\nu(1)}, \bar{c}_{\nu(2)}) \in G_* \times G_*$, where, for $\ell \in \{ 1, 2 \}$, $\bar{c}_{\nu(\ell)} = (c^\nu_{i, \ell} : i < \omega)$. Then we have (recalling Notation \ref{notation_for_crucial_lemma}):
\begin{enumerate}[(a')]
	\item $G_* \subseteq \{ X_{\bar{c}_\nu} : \nu \in \mathcal{F}\}$ (by Fact \ref{description_measure}, (A) of the statement, and (a) above);
	\item $\bar{c}_\nu \in \mathfrak{C}$ (by Fact \ref{description_measure} and (b) above).
	\end{enumerate}
Hence, by (a') and (b'), we have that $\{ X_{\bar{c}_\nu} : \nu \in \mathcal{F}\}$ is a witness for $\mathfrak{fm}(G_*) \leq \lambda$.
\end{proof}

	\begin{lemma}\label{cruc_lemma2} Recalling Notation \ref{notation_solution}(\ref{f}), a sufficient condition for $\mathrm{non}(\mathcal{N}) > \lambda$ (cf. Definition \ref{non_null}) is:
\begin{enumerate}[$(\star)_2$]
	\item for every $Y \subseteq \prod_{i < \omega} m^{**}_i$ of cardinality $\leq \lambda$ there is $\nu$ such that:
\begin{enumerate}[(a)]
	\item $\nu \in \prod_{i < \omega} [m^{**}_i]^{m^*_i}$;
	\item if $\eta \in Y$, then, for infinitely many $i < \omega$, we have that $\eta(i) \in \nu(i)$.
\end{enumerate}
\end{enumerate}
\end{lemma}

	\begin{proof} This is because denoting by $\mu$ (resp. $\mu^*$) the Lebesgue measure (resp. the outer Lebesgue measure) of the Polish space $\prod_{i < \omega} m^{**}_i$
we have that:
$$
\begin{array}{rcl}\label{equation}
\mu^*(Y)  & \leq & \mu^*(\underbrace{\{ \eta \in X : \exists^{\infty} i (\eta(i) \in \nu(i) ) \}}_{X_{\infty}}) \;\;\;\;\;\;\;\;\;\;\;\;\;\;\;\;\;\;\;\;\;\;\;\;\;\;\;\;\;\;\;\;\;\;\;\;\;\;\;\;\; \text{[by $(\star)_2(b)$]}\\
		  & \leq & \mu(\bigcap_{n < \omega} \underbrace{\{ \eta \in X : \bigvee_{i \geq n} \eta(i) \in \nu(i) \}}_{X_n}) \;\;\;\;\;\;\;\;\;\;\;\;\;\;\;\;\;\;\;\;\;\;\text{ [$X_{\infty} \subseteq X_n$, $\forall n < \omega$]}\\
		  & \leq & lim_{n \rightarrow \infty} \mu(\{ \eta \in X : \bigvee_{i \geq n} \eta(i) \in \nu(i) \}) \text{ [$X_{n}$ measurable, $X_n \supseteq X_{n+1}$]}\\
		  & \leq & lim_{n \rightarrow \infty} \frac{m^{*}_n}{m^{**}_n} = 0 \hspace{0.015cm}\;\;\;\text{ [cf. Notation \ref{notation_solution}($f$) and properties of $f_2, g_2$ there]}.
\end{array} $$
\end{proof}

	\begin{theorem}\label{cons_th} Assume that $\mathbf{V} \models CH$. Then for some $\aleph_2$-c.c. proper (in fact even cardinal preserving) forcing $\mathbb{P}$ we have that in $\mathbf{V}[\mathbb{P}]$ both of the conditions below are satisfied:
\begin{enumerate}[(a)]
	\item the statement $(\star)_1$ from Lemma \ref{cruc_lemma1} for $\lambda = \aleph_1$;
	\item the statement $(\star)_2$ from Lemma \ref{cruc_lemma2} for $\lambda = \aleph_1$.
\end{enumerate}
\end{theorem}

	\begin{proof} This is by \cite[Theorem 2]{forcing} and the choice of $(f_1, g_1), (f_2, g_2)$ in Notation~\ref{notation_solution}.
\end{proof}

	\begin{proof}[Proof of Main Theorem] This follows from Lemmas \ref{cruc_lemma1} and \ref{cruc_lemma2}, and Theorem \ref{cons_th}.
\end{proof}

\end{document}